\documentclass{article}
\usepackage[utf8]{inputenc}
\usepackage{amssymb,amsmath, amsthm}
\usepackage[utf8]{inputenc} % set input encoding (not needed with XeLaTeX)
\usepackage{tikz-cd}

\usepackage{geometry} % to change the page dimensions
\usepackage{graphicx} % support the \includegraphics command and options

\newcommand{\R}{\mathbb{R}}

\newcommand{\s}{\mathbb{S}}

\newtheorem{thm}{Theorem}
\newtheorem{cor}{Corollary}

\newtheorem{prop}{Proposition}
\newtheorem{lem}{Lemma}

\newcommand*\xbar[1]{%
  \hbox{%
    \vbox{%
      \hrule height 0.5pt % The actual bar
      \kern0.5ex%         % Distance between bar and symbol
      \hbox{%
        \kern-0.1em%      % Shortening on the left side
        \ensuremath{#1}%
        \kern-0.1em%      % Shortening on the right side
      }%
    }%
  }%
}

\usepackage{booktabs} % for much better looking tables
\usepackage{array} % for better arrays (eg matrices) in maths
\usepackage{paralist} % very flexible & customisable lists (eg. enumerate/itemize, etc.)
\usepackage{verbatim} % adds environment for commenting out blocks of text & for better verbatim
\usepackage{subfig} % make it possible to include more than one captioned figure/table in a single float
\usepackage{fancyhdr} % This should be set AFTER setting up the page geometry
\usepackage{sectsty}
\allsectionsfont{\sffamily\mdseries\upshape} % (See the fntguide.pdf for font help)
\usepackage[nottoc,notlof,notlot]{tocbibind} % Put the bibliography in the ToC
\usepackage[titles,subfigure]{tocloft} % Alter the style of the Table of Contents

\title{The Nonexistence of Noncompact Type-I Ancient 3-d $\kappa$-Solutions of Ricci Flow with Positive Curvature}
\author{Max Hallgren}
\date{}

\begin{document}
\maketitle

\noindent \textbf{Abstract}. In this short paper, we show there do not exist three-dimensional noncompact $\kappa$-solutions of Ricci flow that have positive curvature and satisfy a Type-I bound.  This represents progress towards the proof of Perelman's conjecture that the only complete noncompact three-dimensional $\kappa$-solution with positive curvature is the Bryant soliton.

\section{Introduction}

This paper addresses a class of noncompact solutions to Ricci flow essential to the classification of singularities in three-dimensions.  Namely, we classify three-dimensional complete noncompact $\kappa$-solutions $(M^3, g(t))$, $t \in (-\infty, 0)$, that satisfy the Type-I condition
$$|Rm(g(t))|_{g(t)} \leq \dfrac{C}{|t|} \hspace{6 mm} \mbox{ for all } t \in (-\infty, 0) .$$
Here, a $\kappa$ solution means a complete ancient solution of Ricci flow that is $\kappa$-noncollapsed on all scales and has bounded nonnegative curvature.

\begin{prop}
Let $(M^3, g(t)), \: t \in (-\infty, 0]$ be a complete noncompact 3-dimensional ancient Type-I $\kappa$-solution.  Then $(M^3, g(t))$ is a quotient of the shrinking round cylinder.
\end{prop}

\noindent This result was communicated to Lei Ni by Perelman, though Perelman did write down the proof.  The proposition was recently proved independently in \cite{yongjia}, but the proof differs from ours in that it uses crucially a result about backwards stability of necks proven in \cite{klott2}.

Roughly speaking, the idea of the proof is as follows.  We use Perelman's results on the global structure of 3d $\kappa$-solutions to show that, at every time, the $\kappa$-solution looks neck-like outside of a compact subset.  We take backwards limits based in this set to establish a contradiction.  If this subset does not move too quickly as $t \to -\infty$, then Perelman's asymptotic shrinker theorem gives convergence to a cylinder.  If the subset moves quickly, we show its curvature must still be comparable to the maximum curvature of the time slice.  Then we are able to show the set cannot move too quickly, making use of the Type-I distortion estimates and the global structure of the solution.\\

The following is an easy consequence of the proof of Proposition 1, and does not depend on the Type-I assumption.

\begin{prop} The scalar curvature at any soul of $(M, g(t))$ is comparable to that of the maximum curvature of the time slice.
\end{prop}

For a moment, consider a 3d noncompact $\kappa$-solution $(M^3, g(t))_{t \in (-\infty, 0]}$ with PCO.  Ding \cite{ding} showed that no such solution can have a Type-I forwards singularity.  Cao-Chow-Zhang \cite{caoben} proved independently that the solution cannot be both Type-I forwards and Type-I ancient.  It is conjectured that the only complete noncompact $\kappa$-solution with positive curvature is the Bryant Soliton, but for now combining the above results gives the following. 

\begin{cor} Suppose $(M^3, g(t))$, $t \in (-\infty, 0]$ is a noncompact 3d $\kappa$-solution with positive curvature.  Then $(M, g(t))$ is Type-II ancient, and has either a forward Type-II singularity, or is eternal.  In particular, after rescaling, $(M, g(t))$ has forward and backward limits equal to the Bryant soliton.
\end{cor}

Note that, by Hamilton's Type-II rescaling, any noncompact 3d $\kappa$-solution with PCO must have both forward and backwards limits equal to the Bryant soliton, though the point-picking method to achieve the forward limit depends on whether the solution is eternal or suffers a Type-II forward singularity.

The author is thankful to Yongjia Zhang for providing a simplification of Case 3 of the main result.

\section{Preliminaries}

\noindent Throughout, we write $d_t(x, y) := d_{g(t)}(x, y)$ for $x, y \in M$.

Given an ancient solution $(M, g(t))$, $t \in (-\infty, 0]$ of Ricci flow, an evolving $\epsilon$-neck is a subset of the form
$$N := \{ (x, t) \in M \times (-\infty, 0] \: ; \: d_0(x, x_0) < \epsilon^{-\frac{1}{2}} R(x_0, t_0)^{-\frac{1}{2}}, \: t \in (t_0 -\epsilon^{-1}R(x_0, t_0)^{-1}, t_0] \} $$
that is, after rescaling by $R(x_0, t_0)$, $\epsilon$-close in the $C^{|1/\epsilon|}$-topology to the shrinking cylinder $\s^2 \times (-\epsilon^{-1}, \epsilon^{-1}).$  The point $(x_0, t_0)$ is then called the center of $N$.  We denote by $K(\epsilon)$ the set of points which are not the centers of evolving $\epsilon$-necks.

It is essential to our arguments to use the following version of Perelman's result proved by Kleiner and Lott \cite{klott}:

\begin{thm} (Global Structure of noncompact 3d $\kappa$-solutions \cite{klott}) For any $\epsilon >0$, there exists $D = D(\epsilon, \kappa) > 0$ such that if $(M^3, g(t))$, $t \in (-\infty, 0]$ is an noncompact $\kappa$-solution with PCO, then:\\
i. $K(\epsilon)$ is a compact set with 
$$\mbox{diam}_{g(0)}K(\epsilon) < D R(x_0, 0),$$
ii. $D^{-1}R(x, 0) \leq R(x_0, 0) \leq D R(x, 0)$ for any $x \in K(\epsilon)$,\\
where $x_0$ is any soul of $(M, g(0))$.
\end{thm}

\vspace{6 mm}
Next, we note that the scalar curvature at a fixed point in space cannot decay too quickly as $t \to -\infty$

\begin{lem} For any Type-I ancient $\kappa$-solution $(M^3, g(t))$, $t \in (-\infty, 0]$, and any fixed $p \in M$, there exists $c = c(p) >0$ such that $|t|R(p, t) >c$ for all $t \in (-\infty, -1]$.
\end{lem}
\begin{proof} Suppose by way of contradiction that there is a sequence $t_k \to -\infty$ such that $|t_k|R(p, t_k) \to 0$.  We know from \cite{cao} that the sequence $(M, g_k(t), (p, 0))$, where $g_k(t) = |t_k|^{-1}g(|t_k|t)$, converges in the $C^{\infty}$ Cheeger-Gromov sense to a nonflat shrinking soliton $(M_{\infty}^3, g_{\infty}(t), (p_{\infty}, 0))$, $t \in (-\infty, 0]$.  Thus
$$R_k(p, 0) = |t_k|R(p, t_k) \to R(p_{\infty}, 0) = 0,$$
a contradiction.
\end{proof}

\vspace{6 mm}
\noindent Finally, we need distortion estimates to control how fast the cap region can move as $t \to -\infty$.

\begin{lem} For any Type-I ancient $\kappa$-solution $(M^3, g(t))$, $t \in (-\infty, 0]$ satisfying $|t|R(x, t) \leq C$ for all $(x, t) \in M \times (-\infty, 0]$, there exists $C' = C'(C)$ such that for any $x, y \in M$ and $t_1 < t_2 <0$ we have
$$d_{g(t_1)}(x, y) \leq d_{g(t_2)}(x, y) + C'(\sqrt{|t_2|} - \sqrt{|t_1|}) .$$
\end{lem}
\begin{proof}
Fix $x, y \in M$.  Then the global curvature assumption gives
$$\dfrac{\partial}{\partial t} d_{g(t)}(x, y) \geq -4(n-1)\sqrt{C|t|^{-1}} .$$
Integrating from $t_1$ to $t_2$, we get
$$d_{g(t_2)}(x, y) - d_{g(t_1)}(x, y) \geq -8\sqrt{C}(\sqrt{|t_1|} - \sqrt{|t_2|}).$$ 
Moreover, the nonnegative curvature assumption ensures that distances are nonincreasing.
\end{proof}

\section{Proof of the Proposition}

\vspace{4 mm}

\noindent \textit{Proof of Proposition 1} Suppose by way of contradiction that $(M^3, g(t)), t \in (-\infty, 0]$ is not a quotient of the round cylinder.  Then $(M^3, g(t))$ has positive sectional curvature everywhere, so is diffeomorphic to $\R^3$.  Fix $\epsilon > 0$ sufficiently small.  Then by Theorem 1, for each $t \in (-\infty, 0]$ there is a compact subset $K(t) \subseteq M$ such that $M \setminus K(t)$ is the set of points which are the center of evolving $\epsilon$-necks.  Theorem 1 also states that $\mbox{diam}_{g(t)}K(t) \leq DR(y, t)^{-\frac{1}{2}}$ for all $t \in (-\infty, 0)$, $y \in K(t)$.

Fix $p \in M$. For large enough $|t|$, the spacetime point $(p, t)$ is the center of an evolving $\epsilon$-neck: otherwise there is a sequence $t_j \to -\infty$ where $(p, t_j)$ is not the center of an evolving $\epsilon$-neck.  By the Type-I assumption, and applying the definition of reduced length to the path constant in space, we see that the spacetime sequence $(p, t_j)$ has uniformly bounded reduced length with respect to $(p, 0)$.  Thus Perelman's asymptotic shrinker theorem gives that $(M, |t_j|^{-1}g(|t_j|t), (p, -1)), t \in (-\infty, 0)$ converges in the $C^{\infty}$ Cheeger-Gromov sense to a noncompact, nonflat gradient Ricci soliton with bounded curvature.  However, such a soliton is a shrinking cylinder by Perelman's classification, leading to a contradiction.\\

\noindent \textbf{Case 1:} There exists a sequence $(x_k, t_k)$ in $M \times (-\infty, 0]$ with $x_k \in K(t_k)$, $t_k \to -\infty$, and 
$$\liminf_{k \to \infty} \dfrac{d_{g(t_k)}(x_k, p)}{\sqrt{\tau_k}} < \infty .$$
In this case, we have the result of Naber \cite{naber} that
$$l_{(p, 0)}(x_k, \tau_k) \leq A \left( 1 + \dfrac{d_{g(t_k)}(x_k, p)}{\sqrt{\tau_k}} \right)^2$$
for some $A< \infty$ universal.  Thus we have a bound on $l_{(p, 0)}(x_k, \tau_k)$ independent of $k$, so we can apply Perelman's asymptotic shrinker theorem \cite{poincare} to get subconvergence of $(M, |t_k|g(|t_k|t), (x_k, -1))_{t \in (-\infty, -1]}$ to a cylinder, a contradiction.  \\

\noindent \textbf{Case 2:} There exists a sequence $(x_k, t_k)$ in $M \times (-\infty, 0]$ such that $t_k \to -\infty$, $x_k \in K(t_k)$, and
$$\lim_{k \to \infty} |t_k|R(x_k, t_k)  = 0 .$$
In this case, using that $R(p, t_k) \geq c|t_k|$, we get
$$\lim_{k \to \infty} \dfrac{R(x_k, t_k)}{R(p, t_k)} = 0 .$$
Also, note that any soul $y$ of $(M, g(t_k))$ must lie in $K(t_k)$ as long as $\epsilon$ was chosen smaller than some universal constant.  In fact, we will show that $y$ must lie outside the middle two-thirds of any $\epsilon$-neck $N \subseteq (M, g(t))$.  

Suppose by way of contradiction that $y$ is in the middle two-thirds of $N$.  Then there is an open subset $N'$ that is a $10\epsilon$-neck that is disjoint from $y$, whose center sphere separates $M$, and is such that $y$ lies in the unbounded part of $M \setminus S$.  However, by \cite[Lem 2.20]{poincare}, this $10 \epsilon$-neck separates $y$ from the unique end of $M$, contradicting the fact that it is contained in the unbounded component of $M \setminus S$.

It is a fact from the theory \cite[Cor 2.21]{poincare} of noncompact manifolds of positive curvature that there exists $\bar{C} = \bar{C}(\epsilon) < \infty$ universal with the following property: for any $\epsilon$-neck centered at $z_1$ whose center sphere separates $y$ from another $\epsilon$-neck centered at $z_2$, we have $R(z_2, t_k) \leq \bar{C} R(z_1, t_k)$.  Note that, in the statement of $\cite[Cor 2.21]{poincare}$, it is required that the $\epsilon$-necks are disjoint from $y$, but since $y$ lies outside of the middle two-thirds of any $\epsilon$-neck, $y$ is disjoint from the $2\epsilon$-necks centered at $z_1$ and $z_2$, so we may apply the theorem by replacing $\epsilon$ with $2\epsilon$.  We apply this with $z_2 = p$, and with $z_1$ any point of $\partial K(t_k)$, so that $R(z_1, t_k) \leq C R(x_k, t_k)$.  Combining this with Theorem 1 gives
$$R(p, t_k) \leq \bar{C} C R(x_k, t_k) $$
for all $k$, contradicting the above inequality. \\

\noindent \textbf{Case 3:} $$\liminf_{t \to -\infty} \dfrac{d_t(p, K(t))}{\sqrt{|t|}} = \infty , \hspace{6 mm} \liminf_{t \to -\infty}\inf_{x \in K(t)} |t|R(x, t)  \geq b > 0 .$$
In this case we also have
$$\liminf_{t \to -\infty} \inf_{x \in K(t)} R(x, t)d_t^2(x, p) = \infty.$$
Fix $t_0 < 0$ such that $p$ is the center of an evolving $\epsilon$-neck based at $(p, t)$ and such that $|t|R(x, t) \geq \frac{1}{2}b$ for all $t \leq t_0$, $x \in K(t)$.  Then for $t \leq t_0$ we have
$$\mbox{diam}_{g(t)}(K(t)) \leq D (2|t|/b)^{\frac{1}{2}}.$$
Recall the distortion constant $C' = C'(C)<\infty$ from Lemma 3, the constant $c = c(p) >0$ from Lemma 2, and the Elliptic-type constant $A< \infty$ from Lemma 1.  
By assumption, we may choose $t_1 \leq t_0$ such that
$$C^* := \dfrac{d_{t_1}(p, K(t_1))}{\sqrt{|t_1|}} \geq 100,000((D+ A + 1)b^{-\frac{1}{2}} + C') \hspace{6 mm} \mbox{ and }$$
$$\dfrac{d_{25t_1}(p, K(25|t_1|))}{\sqrt{25|t_1|}} \geq C^* .$$
Set $t_2 := 25t_1$, and let $x_i \in K(t_i)$ be a soul of $(M, g(t_i))$.  Then
$$d_{t_2}(x_2, p) \geq 5C^* \sqrt{|t_1|}, $$ 
$$d_{t_2}(x_2, x_1) \geq 5C^*\sqrt{|t_1|} - (C^* \sqrt{|t_1|} + 5C'\sqrt{|t_1|} + 2Db^{-\frac{1}{2}}|t_1|^{\frac{1}{2}}) \geq (3.99)C^*\sqrt{|t_1|} .$$

We claim that the center sphere $S(p)$ of the $\epsilon$-neck centered at $p$ separates $x_2$ from $x_1$.  In fact, since $x_1$ lies outside the center two-thirds of the $\epsilon$-neck $N$ centered at $p$, \cite[Lem A.9]{poincare} implies that $S:= \partial B(x_1, R(p, t_1)^{-\frac{1}{2}} + d(x_1, p)) \cap N$ lies on the unbounded component of the complement of the center sphere $S(p)$ of $N$.  Note that $R(p, t_1)^{-\frac{1}{2}} \leq c^{-\frac{1}{2}}|t_1|^{\frac{1}{2}}$.  Since level sets for $d(x_1, \cdot)$ are connected \cite[p.27]{poincare}, and $S$ is in the middle two-thirds of an $\epsilon$-neck, we must have $S = \partial B(x_1, R(p, t_1)^{-\frac{1}{2}} + d(x_1, p))$.  Because
$$d_{t_1}(x_1, x_2) \geq (4.99)C^* |t_1|^{\frac{1}{2}} \geq c^{-\frac{1}{2}}|t_1|^{\frac{1}{2}} + (C^* + Db^{-\frac{1}{2}})|t_1|^{\frac{1}{2}} + 3C^*|t_1|^{\frac{1}{2}} > R(p, t_1)^{-\frac{1}{2}} + d(x_1, p) + 3C^* |t_1|^{\frac{1}{2}},$$
so any minimal $g(t_1)$-geodesic from $x_1$ to $x_2$ must leave $B(x_1, d(x_1, p) + R(p, t_1)^{-\frac{1}{2}})$, hence it must leave the bounded component of $M \setminus S(p)$.  In particular, $x_2$ is in the unbounded component of $M \setminus S(p)$.

We can take $\epsilon$ small enough so that $\epsilon^{-1} C^{-1} > 25$, so that
$$\epsilon^{-1}R(y, t_1)^{-1} \geq \epsilon^{-1}C^{-1}|t_1| > 25|t_1| .$$
This means that $N \times \{t_2\}$ is a time slice of an evolving $\epsilon$-neck defined on a time interval containing $[t_2, t_1]$.  In particular, $S(p)$ separates $x_2$ from the unique end of $M$ (again by \cite[Lem 2.20]{poincare}), so $x_2$ is in the bounded component of $M \setminus S(p)$, a contradiction.

\flushright $\qed$

\flushleft

\end{document}